\documentclass[12pt]{amsart}

\normalfont
\usepackage[T1]{fontenc}

\usepackage{amssymb}
\usepackage{graphicx}
\usepackage{lscape}
\usepackage{mathrsfs}

\usepackage[curve,tips]{xypic}
\SelectTips{eu}{11} \UseTips

\renewcommand{\epsilon}{\varepsilon}

\newtheorem{theorem}{Theorem}[section]
\newtheorem{prop}[theorem]{Proposition}

\newtheorem{lemma}[theorem]{Lemma}

\newtheorem*{theoremA}{Theorem A}
\newtheorem*{corB}{Corollary B}
\newtheorem*{propC}{Proposition C}
\newtheorem*{theoremD}{Theorem D}
\newtheorem*{theoremE}{Theorem E}
\newtheorem*{propF}{Proposition F}

\theoremstyle{definition}

\newtheorem{definition}[theorem]{Definition}

\theoremstyle{remark}
\newtheorem{remark}[theorem]{Remark}

\newcommand{\cd}{\operatorname{cd}}
\newcommand{\vcd}{\operatorname{vcd}}

\newcommand{\Q}{\mathbb Q}

\newcommand{\Z}{\mathbb Z}
\newcommand{\N}{\mathbb N}

\newcommand{\poset}{|\Lambda(G)|}

\newcommand{\FP}{\operatorname{FP}}

\newcommand{\fpinfty}{{\FP}_{\infty}}

\newcommand{\fp}{\operatorname{FP}}
\newcommand{\ff}{\operatorname{F}}

\newcommand{\F}{{\mathbb F_p}}
\newcommand{\EXT}{\operatorname{Ext}}

\newcommand{\cohom}[3]{H^{{\raise1pt\hbox{$\scriptstyle#1$}}}(#2\>\!,#3)}
\newcommand{\tatecohom}[3]{\widehat H^{{\raise1pt\hbox{$\scriptstyle#1$}}}(#2\>\!,#3)}

\newcommand{\Cohom}[3]%
  {H^{{\raise1pt\hbox{$\scriptstyle#1$}}}\big(#2\>\!,#3\big)}
\newcommand{\Tatecohom}[3]%
  {\widehat H^{{\raise1pt\hbox{$\scriptstyle#1$}}}\big(#2\>\!,#3\big)}

\newcommand{\homol}[3]{H_{{\lower1pt\hbox{$\scriptstyle#1$}}}(#2\>\!,#3)}
\newcommand{\homolog}[2]{H_{{\lower1pt\hbox{$\scriptstyle#1$}}}(#2)}

\newcommand{\COIND}{\operatorname{Coind}}
\newcommand{\IND}{\operatorname{Ind}}

\newcommand{\colim}{\varinjlim}

\renewcommand{\ker}{\operatorname{Ker}}

\newcommand{\mono}{\rightarrowtail}
\newcommand{\epi}{\twoheadrightarrow}

\newcommand{\ra}{\rightarrow}

\newcommand{\eg}{{\underline EG}}
\newcommand{\eh}{{\underline EH}}

\newcommand{\Hom}{\operatorname{Hom}}

\newcommand{\blah}{{\phantom M}}

\newcommand{\Mod}{\mathfrak{Mod}}

\newcommand{\vfp}{\operatorname{VFP}}

\title{When is Group Cohomology Finitary?}
\author{Martin Hamilton}
\address{Department of Mathematics, University of Glasgow, University Gardens, Glasgow G12 8QW, United Kingdom}
\email{m.hamilton@maths.gla.ac.uk}



\begin{document}

\begin{abstract}
  If $G$ is a group, then we say that the functor $H^n(G,-)$ is
  \emph{finitary} if it commutes with all filtered colimit systems
  of coefficient modules. We investigate groups with
  \emph{cohomology almost everywhere finitary}; that is, groups
  with $n$th cohomology functors finitary for all sufficiently
  large $n$. We establish sufficient conditions for a group $G$
  possessing a finite dimensional model for $\eg$ to have
  cohomology almost everywhere finitary. We also prove a stronger
  result for the subclass of groups of finite virtual
  cohomological dimension, and use this to answer a question of
  Leary and Nucinkis. Finally, we show that if $G$ is a locally
  (polycyclic-by-finite) group, then $G$ has cohomology almost
  everywhere finitary if and only if $G$ has finite virtual
  cohomological dimension and the normalizer of every non-trivial
  finite subgroup of $G$ is finitely generated.
\end{abstract}

\maketitle

\section{Introduction}

Let $G$ be a group and $n\in\N$. The $n$th cohomology of $G$ is a
functor $$H^n(G,-):=\EXT^n_{\Z G}(\Z,-)$$ from the category of $\Z
G$-modules to the category of $\Z$-modules, and we say that it is
\emph{finitary} if it commutes with all filtered colimit systems
of coefficient modules (see \S $3.18$ in
\cite{AccessibleCategories}; also \S $6.5$ in \cite{Leinster}).

Brown \cite{BrownPaper} has characterised groups of type
$\fpinfty$ in terms of finitary functors (see also results of
Bieri, Theorem $1.3$ in \cite{Homdim}):

\begin{prop}\label{brown prop}
  A group $G$ is of type $\fpinfty$ if and only if $H^n(G,-)$ is
  finitary for all $n$.
\end{prop}

It seems natural, therefore, to consider groups whose $n$th
cohomology functors are finitary for \emph{almost all} $n$. We say
that such a group has \emph{cohomology almost everywhere
finitary}.

In this paper, we shall investigate groups with cohomology almost
everywhere finitary. We begin with the class of locally
(polycyclic-by-finite) groups, and in \S\ref{sec:Thm A} we prove
the following:

\begin{theoremA}
  Let $G$ be a locally (polycyclic-by-finite) group. Then $G$ has
  cohomology almost everywhere finitary if and only if $G$ has
  finite virtual cohomological dimension and the normalizer of
  every non-trivial finite subgroup of $G$ is finitely generated.
\end{theoremA}

If $G$ is a locally (polycyclic-by-finite) group with cohomology
almost everywhere finitary, then a result of Kropholler (Theorem
$2.1$ in \cite{Continuity:cohomologyfunctors}) shows that $G$ has
a finite dimensional model for the classifying space $\eg$ for
proper actions and, furthermore, that there is a bound on the
orders of the finite subgroups of $G$ (see \S $5$ of
\cite{LearyNucinkis} for a brief explanation of the classifying
space $\eg$). In \S \ref{sec:Lemma 1.2} we prove the following
Lemma:

\begin{lemma}\label{finite vcd lemma}
  Let $G$ be a locally (polycyclic-by-finite) group. Then the
  following are equivalent: \begin{enumerate}
    \item There is a finite dimensional model for $\eg$, and there is a
    bound on the orders of the finite subgroups of $G$;

    \item $G$ has finite virtual cohomological dimension; and

    \item There is a finite dimensional model for $\eg$, and $G$
    has finitely many conjugacy classes of finite subgroups.
  \end{enumerate}
\end{lemma}

Therefore, we can reduce our study to those groups $G$ which have
finite virtual cohomological dimension. We then have the following
short exact sequence: $$N\mono G\epi Q,$$ where $N$ is a
torsion-free, locally (polycyclic-by-finite) group of finite
cohomological dimension, and $Q$ is a finite group. Hence, in
order to prove Theorem A, we must consider three cases. The first
case is when $G$ is torsion-free. In this case, $G$ has finite
cohomological dimension, so $H^n(G,-)=0$, and hence is finitary,
for all sufficiently large $n$. The next simplest case, when $G$
is the direct product $N\times Q$ is treated in \S \ref{sec:
direct prod}, and the general case is then proved in \S
\ref{sec:Thm A}.

Now, if $G$ is any group with cohomology almost everywhere
finitary, and $H$ is a subgroup of $G$ of finite index, then it is
always true that $H$ also has cohomology almost everywhere
finitary (see Lemma \ref{finite index} below). However, in the
case of locally (polycyclic-by-finite) groups we can say much more
than this:

\begin{corB}
  Let $G$ be a locally (polycyclic-by-finite) group. If $G$ has
  cohomology almost everywhere finitary, then every subgroup of
  $G$ also has cohomology almost everywhere finitary.
\end{corB}

This is not true in general, however, as can been seen from
Proposition \ref{cor B not true in general} below.

Next, we consider the class of elementary amenable groups, and
prove the following in \S\ref{sec:El Am}:

\begin{propC}
  Let $G$ be an elementary amenable group with cohomology almost
  everywhere finitary. Then $G$ has finitely many conjugacy
  classes of finite subgroups, and $C_G(E)$ is finitely generated
  for every $E\leq G$ of order $p$.
\end{propC}

In \S\ref{sec:fin vcd} we investigate the class of groups of
finite virtual cohomological dimension. In this section we work
over a ring $R$ of prime characteristic $p$, instead of over $\Z$,
by defining the $n$th cohomology of a group $G$ as
$$H^n(G,-):=\EXT^n_{RG}(R,-).$$ In order to make it clear that we
are now working over $R$, we say that $H^n(G,-)$ is \emph{finitary
over $R$} if and only if the functor $\EXT^n_{RG}(R,-)$ is
finitary. We have an analogue of Proposition \ref{brown prop},
characterising the groups of type $\fpinfty$ over $R$ as those
with $n$th cohomology functors finitary over $R$ for all $n$. We
can similarly define the notion of a group having \emph{cohomology
almost everywhere finitary over $R$}, and we prove the following
result:

\begin{theoremD}
  Let $G$ be a group of finite virtual cohomological dimension,
  and $R$ be a ring of prime characteristic $p$. Then the
  following are equivalent: \begin{enumerate}
    \item $G$ has cohomology almost everywhere finitary over $R$;

    \item $G$ has finitely many conjugacy classes of elementary
    abelian $p$-subgroups and the normalizer of every non-trivial
    elementary abelian $p$-subgroup of $G$ is of type $\fpinfty$
    over $R$; and

    \item $G$ has finitely many conjugacy classes of elementary
    abelian $p$-subgroups and the normalizer of every non-trivial
    elementary abelian $p$-subgroup of $G$ has cohomology almost
    everywhere finitary over $R$.
  \end{enumerate}
\end{theoremD}

We then adapt the proof of Theorem D slightly, and in \S
\ref{sec:L and N} we use it to prove the following result, which
answers a question of Leary and Nucinkis (Question $1$ in
\cite{LearyNucinkis}):

\begin{theoremE}
  Let $G$ be a group of type $\vfp$ over $\F$, and $P$ be a
  $p$-subgroup of $G$. Then the centralizer $C_G(P)$ of $P$ is
  also of type $\vfp$ over $\F$.
\end{theoremE}

Finally, in \S\ref{sec:EG} we return to working over $\Z$, and
consider the class of groups which possess a finite dimensional
model for $\eg$. We prove the following:

\begin{propF}
  Let $G$ be a group which possesses a finite dimensional model
  for the classifying space $\eg$ for proper actions. If
  \begin{enumerate}
    \item $G$ has finitely many conjugacy classes of finite
    subgroups; and

    \item The normalizer of every non-trivial finite subgroup of
    $G$ has cohomology almost everywhere finitary,
  \end{enumerate}

  Then $G$ has cohomology almost everywhere finitary.
\end{propF}

However, the converse of this result is false, and we shall
exhibit counter-examples in \S\ref{sec:EG} by using a theorem of
Leary (Theorem $20$ in \cite{LearyVF}). These counter-examples
show that the converse of Proposition F is false even for the
subclass of groups of finite virtual cohomological dimension.

\subsection{Acknowledgements}

I would like to thank my research supervisor Peter Kropholler for
all of his advice and support throughout this project. I would
also like to thank Ian Leary for useful discussions concerning the
converse of Proposition F, and for suggesting that my results
could be used to prove Theorem E.

\section{The Direct Product Case of Theorem A}\label{sec: direct
prod}

Suppose that $G=N\times Q$, where $N$ is a torsion-free, locally
(polycyclic-by-finite) group of finite cohomological dimension,
and $Q$ is a non-trivial finite group. We wish to show that $G$
has cohomology almost everywhere finitary if and only if the
normalizer of every non-trivial finite subgroup of $G$ is finitely
generated. Now, if $F$ is a non-trivial finite subgroup of $G$,
then $F$ must be a subgroup of $Q$, and so $N$ is a subgroup of
$N_G(F)$ of finite index. Hence, $N_G(F)$ is finitely generated if
and only if $N$ is. It is therefore enough to prove that $G$ has
cohomology almost everywhere finitary if and only if $N$ is
finitely generated.

We begin by assuming that $N$ is finitely generated. Therefore $N$
is polycyclic-by-finite, and hence of type $\fpinfty$ (Examples
$2.6$ in \cite{Homdim}). The property of type $\fpinfty$ is
inherited by supergroups of finite index, so $G$ is also of type
$\fpinfty$. Therefore, by Proposition \ref{brown prop}, we see
that $G$ has cohomology almost everywhere finitary.

For the converse, we shall prove a more general result which does
not place any restrictions on the group $N$. Firstly, we need the
following three lemmas:

\begin{lemma}\label{finite index}
  Let $G$ be a group, and $H$ be a subgroup of finite index. If
  $H^n(G,-)$ is finitary, then $H^n(H,-)$ is also finitary.
\end{lemma}

\begin{proof}
  Suppose that $H^n(G,-)$ is finitary. From Shapiro's Lemma
  (Proposition $6.2$ \S III in \cite{Brown}), we have:
  $$H^n(H,-)\cong H^n(G,\COIND^G_H-).$$ Then, as $H$ has finite
  index in $G$, it follows from Lemma $6.3.4$ in \cite{Wiebel}
  that $\COIND^G_H(-)\cong\IND^G_H(-)$. Therefore, $$H^n(H,-)\cong
  H^n(G,\IND^G_H-)\cong H^n(G,-\otimes_{\Z H}\Z G),$$ and as
  tensor products commute with filtered colimits, we see that $H^n(H,-)$ is
  the composite of two finitary functors, and hence is itself
  finitary.
\end{proof}

\begin{lemma}\label{change of rings lemma}
  Let $G$ be a group, and $R_1\ra R_2$ be a ring homomorphism. If
  $H^n(G,-)$ is finitary over $R_1$, then $H^n(G,-)$ is finitary
  over $R_2$.
\end{lemma}

\begin{proof}
  We see from Chapter $0$ of \cite{Homdim} that for any
  $R_2G$-module $M$ we have the following isomorphism:
  $$\EXT^n_{R_2G}(R_2,M)\cong\EXT^n_{R_1G}(R_1,M),$$ where $M$ is
  viewed as an $R_1G$-module via the homomorphism $R_1\ra R_2$.
  The result now follows.
\end{proof}

\begin{lemma}\label{direct sum fin implies each term fin}
  Let $F_1,F_2:\Mod_R\ra\Mod_S$, and suppose that $F$ is the direct sum of $F_1$ and $F_2$. If $F$ is
  finitary, then so are $F_1$ and $F_2$.
\end{lemma}

\begin{proof}
  As $F$ is the direct sum of $F_1$ and $F_2$, we have the
  following exact sequence of functors: $$0\ra F_1\ra F\ra F_2\ra
  0.$$ Let $(M_{\lambda})$ be a filtered colimit system of
  $R$-modules. We have the following commutative diagram with
  exact rows: $$\xymatrix{\colim_{\lambda}
  F_1(M_{\lambda})\ar@{>->}[r]\ar[d]_{f_1} & \colim_{\lambda}
  F(M_{\lambda})\ar@{->>}[r]\ar[d]^{f} & \colim_{\lambda}
  F_2(M_{\lambda})\ar[d]^{f_2} \\ F_1(\colim_{\lambda}
  M_{\lambda})\ar@{>->}[r] & F(\colim_{\lambda}
  M_{\lambda})\ar@{->>}[r] & F_2(\colim_{\lambda} M_{\lambda})}$$ As $F$ is finitary, we see that the map $f$ is an isomorphism.
  It then follows from the Snake Lemma that $f_1$ is a
  monomorphism and $f_2$ is an epimorphism.

  Now, as $F$ is the direct sum of $F_1$ and $F_2$, we also have the
  following exact sequence of functors: $$0\ra F_2\ra F\ra F_1\ra
  0,$$ and hence the following commutative diagram with exact
  rows: $$\xymatrix{\colim_{\lambda}
  F_2(M_{\lambda})\ar@{>->}[r]\ar[d]_{f_2} & \colim_{\lambda}
  F(M_{\lambda})\ar@{->>}[r]\ar[d]^{f} & \colim_{\lambda}
  F_1(M_{\lambda})\ar[d]^{f_1} \\ F_2(\colim_{\lambda}
  M_{\lambda})\ar@{>->}[r] & F(\colim_{\lambda}
  M_{\lambda})\ar@{->>}[r] & F_1(\colim_{\lambda} M_{\lambda})}$$
  and a similar argument to above shows that $f_2$ is a monomorphism and $f_1$ is an
  epimorphism. The result now follows.
\end{proof}

\begin{prop}\label{direct product prop}
  Let $Q$ be a non-trivial finite group, and $N$ be any group. If there is some natural number $k$ such that $H^k(N\times Q,-)$
  is finitary, then $N$ is finitely generated.
\end{prop}

\begin{proof}
  Suppose that $H^k(N\times Q,-)$ is finitary. As $Q$ is a non-trivial finite group,
  we can choose a subgroup $E$ of $Q$ of order $p$, for some prime $p$, so
  $N\times E$ is a subgroup of $N\times Q$ of finite index. It
  then follows from Lemma \ref{finite index} that $H^k(N\times E,-)$ is also
  finitary. Then, by Lemma \ref{change of rings lemma}, we see that $H^k(N\times E,-)$ is finitary over $\F$.

  Let $M$ be any $\F N$-module, and $\F$ be the trivial $\F
  E$-module. Applying the K\"{u}nneth Theorem gives the
following isomorphism: $$\begin{array}{lcl} H^k(N\times E,M) &
\cong & \bigoplus_{i+j=k} H^i(N,M)\otimes_{\F} H^j(E,\F)\\
\blah & \cong & \bigoplus_{i=0}^k H^i(N,M),\end{array}$$ and as
this holds for any $\F N$-module $M$, we have an isomorphism of
functors for modules on which $E$ acts trivially. Then, as
$H^k(N\times E,-)$ is finitary over $\F$, it follows from Lemma
\ref{direct sum fin implies each term fin} that $H^0(N,-)$ is also
finitary over $\F$. It then follows that $N$ is finitely generated
(see, for example, Proposition $2.1$ in \cite{Homdim}).

\end{proof}

The converse of the direct product case now follows immediately.

\section{Proof of Theorem A}\label{sec:Thm A}

Let $G$ be a locally (polycyclic-by-finite) group of finite
virtual cohomological dimension. We begin with the following
useful result of Cornick and Kropholler (Theorem A in
\cite{stronggpgradedring}):

\begin{prop}\label{finite projective dimension}

Let $G$ be a group possessing a finite dimensional model for
$\eg$, and $M$ be an $RG$-module. Then $M$ has finite projective
dimension over $RG$ if and only if $M$ has finite projective
dimension over $RH$ for all finite subgroups $H$ of $G$.

\end{prop}

\begin{theorem}\label{split infinitary result}

Let $G$ be a locally (polycyclic-by-finite) group of finite
virtual cohomological dimension. If $G$ has cohomology almost
everywhere finitary, then the normalizer of every non-trivial
finite subgroup of $G$ is finitely generated.

\end{theorem}

\begin{proof}

Let $F$ be a non-trivial finite subgroup of $G$, so we can choose
a subgroup $E$ of $F$ of order $p$, for some prime $p$. As $G$ has
finite virtual cohomological dimension, it has a torsion-free
normal subgroup $N$ of finite index. Let $H:=NE$, so it follows
from Lemma \ref{finite index} that $H$ has cohomology almost
everywhere finitary.

Let $\Lambda$ denote the set of non-trivial finite subgroups of
$H$, so $\Lambda$ consists of subgroups of order $p$. Now $H$ acts
on this set by conjugation, so the stabilizer of any $K\in\Lambda$
is $N_H(K)$. Also, for each $K\in\Lambda$, we see that the set of
$K$-fixed points $\Lambda^K$ is simply the set $\{K\}$, because if
$K$ fixed some $K'\neq K$, then $KK'$ would be a subgroup of $H$
of order $p^2$, which is a contradiction.

We have the following short exact sequence:
$$J\mono\Z\Lambda\stackrel{\varepsilon}{\epi}\Z,$$ where
$\varepsilon$ denotes the augmentation map. For each
$K\in\Lambda$, we see that $J$ is free as a $\Z K$-module with
basis $\{K'-K:K'\in\Lambda\}$. Now, as $H$ has finite virtual
cohomological dimension, it has a finite dimensional model for
$\eh$ (Exercise \S VIII.$3$ in \cite{Brown}), and so it follows
from Proposition \ref{finite projective dimension} that $J$ has
finite projective dimension over $\Z H$. Now, the short exact
sequence $J\mono\Z\Lambda\epi\Z$ gives rise to a long exact
sequence in cohomology, and as $J$ has finite projective
dimension, we conclude that for all sufficiently large $n$ we have
the following isomorphism:
$$H^n(H,-)\cong\EXT^n_{\Z H}(\Z\Lambda,-).$$ Next, as $H$ acts on
$\Lambda$, we can split $\Lambda$ up into its $H$-orbits, so
$$\Lambda =\coprod_{K\in\mathscr{C}}H_K\backslash
H=\coprod_{K\in\mathscr{C}} N_H(K)\backslash H,$$ where $K$ runs
through a set $\mathscr{C}$ of representatives of conjugacy
classes of non-trivial finite subgroups of $H$. This gives the
following isomorphism:
$$\begin{array}{rcl}
  H^n(H,-) & \cong & \prod_{K\in\mathscr{C}}\EXT^n_{\Z H}(\Z[N_H(K)\backslash H],-) \\
  \blah & \cong & \prod_{K\in\mathscr{C}} H^n(N_H(K),-), \\
\end{array}$$ where the last isomorphism follows from the Eckmann--Shapiro
Lemma. Therefore, if $H^n(H,-)$ is finitary, it follows from Lemma
\ref{direct sum fin implies each term fin} that $H^n(N_H(E),-)$ is
also finitary. Hence, as $H$ has cohomology almost everywhere
finitary, we conclude that $N_H(E)$ also has cohomology almost
everywhere finitary.

Now, as $E$ is a finite group,
$$|N_H(E):C_H(E)|<\infty,$$ and so by Lemma \ref{finite index}
we see that $$C_H(E)\cong E\times C_N(E)$$ has cohomology almost
everywhere finitary. It then follows from Proposition \ref{direct
product prop} that $C_N(E)$ is finitely generated, and hence
polycyclic-by-finite.

Now, as $E\leq F$, it follows that $C_N(F)\leq C_N(E)$ and as
every subgroup of a polycyclic-by-finite group is finitely
generated, we see that $C_N(F)$ is finitely generated.

Finally, as $N$ is a subgroup of $G$ of finite index, it follows
that
$$|C_G(F):C_N(F)|<\infty,$$ and so $C_G(F)$ is
finitely generated. Hence $N_G(F)$ is finitely generated, as
required.

\end{proof}

In the remainder of this section we shall prove the converse.
Firstly, we need the following definition from \cite{Deltapaper}:

\begin{definition}\label{definition of
delta}

Let $G$ be a group, and let $\Lambda(G)$ denote the poset of the
non-trivial finite subgroups of $G$. We can view this poset as a
$G$-simplicial complex, which we shall denote by $\poset$, by the
following method: An $n$-simplex in $|\Lambda(G)|$ is determined
by each strictly increasing chain
$$H_0< H_1<\cdots <H_n$$ of $n+1$ non-trivial finite subgroups of $G$.
The action of $G$ on the set of non-trivial finite subgroups
induces an action of $G$ on $|\Lambda(G)|$, so that the stabilizer
of a simplex is an intersection of normalizers; in the case of the
simplex determined by the chain of subgroups above, the stabilizer
is
$$\bigcap_{i=0}^n N_G(H_i).$$ This complex has the property that,
for any non-trivial finite subgroup $K$ of $G$, the $K$-fixed
point complex $\poset^K$ is contractible (for a proof of this, see
Lemma $2.1$ in \cite{Deltapaper}).
\end{definition}

Next, we need the following two results of Kropholler and Mislin:

\begin{prop}\label{n-1 connected Y}

Let $Y$ be a $G$-CW-complex of finite dimension $n$. Then $Y$ can
be embedded into an $n$-dimensional $G$-CW-complex $\widetilde{Y}$
which is $(n-1)$-connected in such a way that $G$ acts freely
outside $Y$.

\end{prop}

\begin{proof}
  This is Lemma $4.4$ of \cite{Deltapaper}. We can take
  $\widetilde{Y}$ to be the $n$-skeleton of the join
  $$Y\ast\underbrace{G\ast\cdots\ast G}_n.$$
\end{proof}

\begin{prop}\label{why is reduced homology projective}

Let $Y$ be an $n$-dimensional $G$-CW-complex which is
$(n-1)$-connected, for some $n\geq 0$. Suppose that $Y^K$ is
contractible for all non-trivial finite subgroups $K$ of $G$. Then
the $n$th reduced homology group $\widetilde{H_n}(Y)$ is
projective as a $\Z K$-module for all finite subgroups $K$ of $G$.

\end{prop}

\begin{proof}
  This is Proposition $6.2$ of \cite{Deltapaper}.
\end{proof}

Finally, we require the following two lemmas:

\begin{lemma}\label{fin functors implies direct sum fin}

Let $$0\ra F_1\ra F \ra F_2\ra 0$$ be an exact sequence of
functors from $\Mod_R$ to $\Mod_S$. If $F_1$ and $F_2$ are
finitary, then so is $F$.

\end{lemma}

\begin{proof}

Let $(M_{\lambda})$ be a filtered colimit system of $R$-modules.
We have the following commutative diagram:
$$\xymatrix{0\ar[r] & \colim_{\lambda}
F_1(M_{\lambda})\ar[r]\ar[d]^{f_1} & \colim_{\lambda}
F(M_{\lambda})\ar[r]\ar[d]^{f} & \colim_{\lambda}
F_2(M_{\lambda})\ar[r]\ar[d]^{f_2} & 0\\0\ar[r] &
F_1(\colim_{\lambda} M_{\lambda})\ar[r] & F(\colim_{\lambda}
M_{\lambda})\ar[r] & F_2(\colim_{\lambda} M_{\lambda})\ar[r] &
0}$$ Now, as $F_1$ and $F_2$ are finitary, the maps $f_1$ and
$f_2$ are isomorphisms. It then follows from the Five Lemma that
$f$ is an isomorphism, and we conclude that $F$ is finitary.

\end{proof}

\begin{lemma}\label{exact seq of aef mods lemma}

Let $G$ be a group. If we have an exact sequence of $\Z G$-modules
$$0\ra A_r\ra A_{r-1}\ra\cdots\ra A_0\ra \Z\ra 0$$
such that, for each $i=0,\ldots,r$, the functor $\EXT^*_{\Z
G}(A_i,-)$ is finitary in all sufficiently high dimensions, then
$G$ has cohomology almost everywhere finitary.

\end{lemma}

\begin{proof}

If $r=0$, then the result follows immediately. Assume, therefore,
that $r\geq 1$, and proceed by induction.

If $r=1$, then we have the short exact sequence $$A_1\mono A_0\epi
\Z$$ which gives the following long exact sequence:
$$\begin{array}{lrrrrrr}
   \cdots\ra\EXT^j_{\Z G}(A_0,-)\ra\EXT^j_{\Z G}(A_1,-)\ra
H^{j+1}(G,-)\ra & \blah & \blah & \blah & \blah & \blah & \blah
\end{array}$$
$$\begin{array}{llllr} \blah & \blah & \blah & \blah & \ra\EXT^{j+1}_{\Z G}(A_0,-)\ra\EXT^{j+1}_{\Z G}(A_1,-)\ra\cdots
  \end{array}$$ and as both $\EXT^*_{\Z G}(A_0,-)$ and $\EXT^*_{\Z G}(A_1,-)$ are
finitary in all sufficiently high dimensions, it follows from the
Five Lemma that $G$ has cohomology almost everywhere finitary.

Now suppose that we have shown this for $r-1$, and that we have an
exact sequence $$0\ra A_r\ra A_{r-1}\ra\cdots\ra A_0\ra \Z\ra 0$$
such that, for each $i=0,\ldots,r$, the functor $\EXT^*_{\Z
G}(A_i,-)$ is finitary in all sufficiently high dimensions. Let
$K:=\ker(A_{r-2}\ra A_{r-3})$, so we have the short exact sequence
$$A_r\mono A_{r-1}\epi K,$$ and an argument similar to above shows that
$\EXT^*_{\Z G}(K,-)$ is finitary in all sufficiently high
dimensions. We then have the following exact sequence:
$$0\ra K\ra A_{r-2}\ra\cdots\ra A_0\ra \Z\ra 0,$$ and the result
now follows by induction.

\end{proof}

Finally, we can now prove the converse:

\begin{theorem}\label{finitary nilpotent
result}

Let $G$ be a locally (polycyclic-by-finite) group of finite
virtual cohomological dimension. If the normalizer of every
non-trivial finite subgroup of $G$ is finitely generated, then $G$
has cohomology almost everywhere finitary.

\end{theorem}

\begin{proof}

Let $\Lambda(G)$ be the poset of all non-trivial finite subgroups
of $G$, and let $\poset$ denote its realization as a
$G$-simplicial complex. As $G$ has finite virtual cohomological
dimension, there is a bound on the orders of its finite subgroups,
and so $\poset$ is finite-dimensional, say $\dim\poset=r$. From
Proposition \ref{n-1 connected Y}, we can embed $\poset$ into an
$r$-dimensional $G$-CW-complex $Y$ which is $(r-1)$-connected,
such that $G$ acts freely outside $\poset$. Consider the augmented
cellular chain complex of $Y$. As $Y$ is $(r-1)$-connected, it has
trivial homology except in dimension $r$, and so we have the
following exact sequence:
$$0\ra\widetilde{H}_r(Y)\ra C_r(Y)\ra\cdots\ra C_0(Y)\ra\Z\ra 0.$$

In order to show that $G$ has cohomology almost everywhere
finitary, it is enough by Lemma \ref{exact seq of aef mods lemma}
to show that the functors $\EXT^*_{\Z G}(\widetilde{H}_r(Y),-)$
and $\EXT^*_{\Z G}(C_l(Y),-)$, $0\leq l\leq r$, are finitary in
all sufficiently high dimensions.

Firstly, notice that for every non-trivial finite subgroup $K$ of
$G$, $Y^K=\poset^K$, as the copies of $G$ that we have added in
the construction of $Y$ have free orbits, and so have no fixed
points under $K$. Thus, $Y$ is an $r$-dimensional $G$-CW-complex
which is $(r-1)$-connected, such that $Y^K$ is contractible for
every non-trivial finite subgroup $K$ of $G$. It then follows from
Proposition \ref{why is reduced homology projective} that
$\widetilde{H}_r(Y)$ is projective as a $\Z K$-module for all
finite subgroups $K$ of $G$. Then by Proposition \ref{finite
projective dimension}, $\widetilde{H}_r(Y)$ has finite projective
dimension over $\Z G$,
 and so $\EXT^n_{\Z G}(\widetilde{H}_r(Y),-)=0$, and thus is
finitary, for all sufficiently large $n$.

Next, for each $0\leq l\leq r$, consider the functor $\EXT^*_{\Z
G}(C_l(Y),-)$. Provided that $n\geq 1$, we see that
$$\EXT^n_{\Z G}(C_l(Y),-)\cong\EXT^n_{\Z G}(C_l(\poset),-)$$ as the copies of $G$ that we have added
in the construction of $Y$ have free orbits, and so the
free-abelian group on them is a free module. Now,
$$\EXT^n_{\Z G}(C_l(\poset),-)\cong\EXT^n_{\Z G}(\Z\poset_l,-),$$
where $\poset_l$ consists of all the $l$-simplicies
$$K_0<K_1<\cdots<K_l$$ in $\poset$. As $G$ acts on $\poset_l$,
we can therefore split $\poset_l$ up into its $G$-orbits, where
the stabilizer of such a simplex is $\bigcap_{i=0}^l N_G(K_i)$. We
then obtain the following isomorphism:

$$\begin{array}{rcl}
  \EXT^n_{\Z G}(\Z\poset_l,-) & \cong & \EXT^n_{\Z
  G}(\Z[\coprod_{\mathscr{C}} \bigcap_{i=0}^l N_G(K_i)\backslash
  G],-) \\ \blah & \cong &  \prod_{\mathscr{C}}\EXT^n_{\Z G}(\Z[\bigcap_{i=0}^l N_G(K_i)\backslash G],-) \\
  \blah & \cong & \prod_{\mathscr{C}} H^n(\bigcap_{i=0}^l N_G(K_i),-), \\
\end{array}$$ where the product is taken over a set $\mathscr{C}$ of representatives of conjugacy classes of non-trivial
finite subgroups of $G$. Now, as $G$ has finite virtual
cohomological dimension, it follows from Lemma \ref{finite vcd
lemma} that there are only finitely many conjugacy classes of
finite subgroups, and so this product is finite.

Now, for each $l$-simplex $K_0<\cdots<K_l$ we have
$$\bigcap_{i=0}^l N_G(K_i)\leq N_G(K_l).$$ Then, as $N_G(K_l)$
is finitely generated, it follows that $\bigcap_{i=0}^l N_G(K_i)$
is also finitely generated, and hence polycyclic-by-finite.

Therefore, $\bigcap_{i=0}^l N_G(K_i)$ is of type $\fpinfty$, and
so by Proposition \ref{brown prop}, $H^n(\bigcap_{i=0}^l
N_G(K_i),-)$ is finitary. Thus $\EXT^n_{\Z G}(C_l(Y),-)$ is
isomorphic to a finite product of finitary functors, and hence by
Lemma \ref{fin functors implies direct sum fin} is finitary. As
this holds for all $n\geq 1$, we see that $\EXT^*_{\Z
G}(C_l(Y),-)$ is finitary in all sufficiently high dimensions,
which completes the proof.
\end{proof}

\section{Proof of Corollary B}\label{sec:Cor B}

\begin{corB}
  Let $G$ be a locally (polycyclic-by-finite) group. If $G$ has
  cohomology almost everywhere finitary, then every subgroup of
  $G$ also has cohomology almost everywhere finitary.
\end{corB}

\begin{proof}
As $G$ has cohomology almost everywhere finitary, it follows from
Theorem A that $G$ has finite virtual cohomological dimension and
the normalizer of every non-trivial finite subgroup of $G$ is
finitely generated.

Let $H$ be any subgroup of $G$, so $$\vcd H\leq\vcd G<\infty.$$
Also, let $F$ be a non-trivial finite subgroup of $H$. Then
$N_G(F)$ is finitely generated, hence polycyclic-by-finite, and as
$$N_H(F)\leq N_G(F),$$ we see that $N_H(F)$ is also finitely
generated. Therefore, we conclude from Theorem A that $H$ has
cohomology almost everywhere finitary.
\end{proof}

This result does not hold in general, however, as the following
proposition shows:

\begin{prop}\label{cor B not true in general}
  Let $G$ be a group of type $\fpinfty$ which has an infinitely
  generated subgroup $H$, and let $Q$ be a non-trivial finite group. Then
  $G\times Q$ has cohomology almost everywhere finitary, but
  $H\times Q$ does not.
\end{prop}

\begin{proof}
  As $G$ is of type $\fpinfty$, it follows that $G\times Q$ is also
  of type $\fpinfty$, and so has cohomology almost everywhere
  finitary. However, as $H$ is infinitely generated, it follows
  from Proposition \ref{direct product prop} that $H^n(H\times
  Q,-)$ is not finitary for any $n$.
\end{proof}

\begin{remark}
  Let $G$ be the free group on two generators $x,y$, so
  $G$ is of type $\fpinfty$ (Example 2.6 in
  \cite{Homdim}), and let $H$ be the subgroup of $G$ generated by
  $y^nxy^{-n}$ for all $n$.  We then
  have a counter-example showing that Corollary B does not
  hold in general.

\end{remark}

\section{A Result on Elementary Amenable Groups}\label{sec:El Am}

\begin{propC}
  Let $G$ be an elementary amenable group with cohomology almost
  everywhere finitary. Then $G$ has finitely many conjugacy
  classes of finite subgroups, and $C_G(E)$ is finitely generated for
  every $E\leq G$ of order $p$.
\end{propC}

\begin{proof}
  Let $G$ be an elementary amenable group with cohomology almost
  everywhere finitary. Kropholler's Theorem (Theorem $2.1$ in
  \cite{Continuity:cohomologyfunctors}) applies to a large class
  of groups, which includes all elementary amenable groups. This
  theorem implies that $G$ has a finite dimensional model for
  $\eg$, and that $G$ has a bound on the orders of its finite
  subgroups. The proof of Lemma \ref{finite vcd lemma} generalizes immediately to the elementary amenable case,
  and we conclude that $G$ has finitely many conjugacy classes of
  finite subgroups, and furthermore that $G$ has finite virtual cohomological
  dimension. Therefore, we can choose a torsion-free normal subgroup $N$
  of $G$ of finite index.

  Let $E$ be any subgroup of $G$ of order $p$, and let $H:=NE$.
  Following the proof of Theorem \ref{split infinitary result}, we
  see that $N_H(E)$ has cohomology almost everywhere finitary.
  Hence, $$C_H(E)\cong E\times C_N(E)$$ also has cohomology almost
  everywhere finitary, and so by Proposition \ref{direct product
  prop} we see that $C_N(E)$ is finitely generated. The result now follows.
\end{proof}

\section{Generalization to Groups of Finite Virtual Cohomological
Dimension}\label{sec:fin vcd}

In this section, we shall prove Theorem D. It suffices to show
this for the case $R=\F$, by the following lemma:

\begin{lemma}
  Let $G$ be a group, and $R$ be a ring of prime characteristic $p$. Then
  $H^n(G,-)$ is finitary over $R$ if and only if $H^n(G,-)$ is
  finitary over $\F$.
\end{lemma}

\begin{proof}
  If $H^n(G,-)$ is finitary over $\F$, then it follows from Lemma
  \ref{change of rings lemma} that $H^n(G,-)$ is finitary over $R$.

  Conversely, suppose that $H^n(G,-)$ is finitary over $R$; that
  is, the functor $\EXT^n_{RG}(R,-)$ is finitary. Let
  $(M_{\lambda})$ be a filtered colimit system of $\F G$-modules.
  Then $(M_{\lambda}\otimes_{\F}R)$ is a filtered colimit system
  of $RG$-modules, and so the
  natural map $$\colim_{\lambda}
  \EXT^n_{RG}(R,M_{\lambda}\otimes_{\F}R)\ra\EXT^n_{RG}(R,\colim_{\lambda}
  M_{\lambda}\otimes_{\F}R)$$ is an isomorphism. Now, as an
  $\F$-vector space, $R\cong\F\oplus V$ for some $\F$-vector space
  $V$. Therefore, for each $\lambda$,
  $$M_{\lambda}\otimes_{\F}R\cong M_{\lambda}\oplus
  M_{\lambda}\otimes_{\F}V,$$ and so
  $$\EXT^n_{RG}(R,M_{\lambda}\otimes_{\F}
  R)\cong\EXT^n_{RG}(R,M_{\lambda})\oplus\EXT^n_{RG}(R,M_{\lambda}\otimes_{\F}
  V).$$ It then follows from Lemma \ref{direct sum fin implies each term fin}
  that the natural map
  $$\colim_{\lambda}\EXT^n_{RG}(R,M_{\lambda})\ra\EXT^n_{RG}(R,\colim_{\lambda}
  M_{\lambda})$$ is an isomorphism. Now, we see from Chapter $0$ of \cite{Homdim} that
  $$\EXT^n_{RG}(R,-)\cong\EXT^n_{\F G}(\F,-)$$ on $\F G$-modules,
  so therefore it follows that the natural map $$\colim_{\lambda}\EXT^n_{\F G}(\F, M_{\lambda})\ra\EXT^n_{\F
  G}(\F,\colim_{\lambda} M_{\lambda})$$ is an isomorphism, and hence that $H^n(G,-)$ is finitary over $\F$.
\end{proof}

\subsection{Proof of (i) $\Rightarrow$ (ii)}$ $

Let $G$ be a group of finite virtual cohomological dimension with
cohomology almost everywhere finitary over $\F$. We begin by
showing that the normalizer of every non-trivial elementary
abelian $p$-subgroup of $G$ is of type $\fpinfty$ over $\F$. In
fact, we shall show that the normalizer of every non-trivial
finite $p$-subgroup of $G$ is of type $\fpinfty$ over $\F$.

\begin{lemma}\label{direct product vcd case}
  Let $N$ be any group, and $Q$ be a non-trivial finite group whose order is
  divisible by $p$. If $N\times Q$ has cohomology almost
  everywhere finitary over $\F$, then $N\times Q$ is of type
  $\fpinfty$ over $\F$.
\end{lemma}

\begin{proof}
  Suppose that $N\times Q$ is not of type $\fpinfty$ over $\F$, so
  $N$ is not of type $\fpinfty$ over $\F$. Therefore, there is
  some $n$ such that $H^n(N,-)$ is not finitary over $\F$.

  Let $E$ be a subgroup of $Q$ of order $p$, so by an argument
  similar to the proof of Proposition \ref{direct product prop} we
  obtain, for each $m$, the following isomorphism of functors: $$H^m(N\times
  E,-)\cong\bigoplus_{i=0}^m H^i(N,-),$$ for modules on which $E$
  acts trivially.

  As $H^n(N,-)$ is not finitary over $\F$, it follows
  from Lemma \ref{direct sum fin implies each term fin} that
  $H^m(N\times E,-)$ is not finitary over $\F$ for all $m\geq n$. Therefore, by an easy generalization of Lemma
  \ref{finite index}, we see that $H^m(N\times Q,-)$
  is not finitary over $\F$ for all $m\geq n$, which is a
  contradiction.
\end{proof}

\begin{lemma}\label{infinitary vcd result}

Let $G$ be a group of finite virtual cohomological dimension with
cohomology almost everywhere finitary over $\F$, and let $E$ be a
subgroup of order $p$. Then the normalizer $N_G(E)$ of $E$ is of
type $\fpinfty$ over $\F$.

\end{lemma}

\begin{proof}

As $G$ has finite virtual cohomological dimension, we can choose a
torsion-free normal subgroup $N$ of finite index. Let $H:=NE$. A
slight variation on the proof of Theorem \ref{split infinitary
result} shows that $N_H(E)$ has cohomology almost everywhere
finitary over $\F$. Therefore,
$$C_H(E)\cong E\times C_N(E)$$ has cohomology almost everywhere
finitary over $\F$, and by Lemma \ref{direct product vcd case}, we
see that $C_H(E)$ is of type $\fpinfty$ over $\F$. Thus, $N_G(E)$
is of type $\fpinfty$ over $\F$, as required.

\end{proof}

\begin{theorem}\label{p-subgroup
case}

Let $G$ be a group of finite virtual cohomological dimension with
cohomology almost everywhere finitary over $\F$, and let $F$ be a
non-trivial finite $p$-subgroup. Then the normalizer $N_G(F)$ of
$F$ is of type $\fpinfty$ over $\F$.

\end{theorem}

\begin{proof}

Suppose that $F$ has order $p^k$, where $k\geq 1$. We proceed by
induction on $k$.

If $k=1$, then the result follows from Lemma \ref{infinitary vcd
result}.

Suppose now that $k\geq 2$. As the centre $\zeta(F)$ of $F$ is
non-trivial, we can choose a subgroup $E\leq\zeta(F)$ of order
$p$. Then $C_G(E)$ is of type $\fpinfty$ over $\F$ by Lemma
\ref{infinitary vcd result}, and Proposition $2.7$ in
\cite{Homdim} shows that $C_G(E)/E$ is also of type $\fpinfty$
over $\F$. By induction, the normalizer of $F/E$ in $C_G(E)/E$,
which is
$$(N_G(F)\cap C_G(E))/E,$$ is of type $\fpinfty$ over $\F$.
Another application of Proposition $2.7$ in \cite{Homdim} shows
that $N_G(F)\cap C_G(E)$ is of type $\fpinfty$ over $\F$, and as
$$C_G(F)\leq N_G(F)\cap C_G(E)\leq N_G(F),$$ we see that
$N_G(F)$ is of type $\fpinfty$ over $\F$.

\end{proof}

Next, we shall show that $G$ has finitely many conjugacy classes
of elementary abelian $p$-subgroups. Firstly, we need the
following lemma:

\begin{lemma}\label{f.d. vect sp lemma}
  Let $G$ be a group. If $H^n(G,-)$ is finitary over
  $\F$, then $H^n(G,\F)$ is finite-dimensional as an $\F$-vector
  space.
\end{lemma}

\begin{proof}
  Suppose that $H^n(G,\F)$ is infinite-dimensional as an
  $\F$-vector space. By the Universal Coefficient Theorem, we have the following isomorphism:
  $$H^n(G,\F)\cong\Hom_{\F}(H_n(G,\F),\F).$$ Hence $H_n(G,\F)$ is
  also infinite-dimensional as an $\F$-vector space, with basis $\{e_i:i\in I\}$, say.
  We then have: $$H^n(G,\F)\cong\prod_I \F.$$

  Next, let $\bigoplus_J \F$ be an infinite direct sum of copies
  of $\F$. As $H^n(G,-)$ is finitary over $\F$, the natural map
  $$\bigoplus_J H^n(G,\F)\ra H^n(G,\bigoplus_J\F)$$ is an
  isomorphism; that is, $$\bigoplus_J\prod_I
  \F\cong\prod_I\bigoplus_J\F,$$ which is clearly a contradiction.
\end{proof}

Next, recall the following definition from \cite{Hennpaper}:

\begin{definition}
  A homomorphism $\phi:A\ra B$ of $\F$-algebras is called a
  \emph{uniform $F$-isomorphism} if and only if there exists a natural number
  $n$ such that:\begin{itemize}
    \item If $x\in\ker\phi$, then $x^{p^n}=0$; and

    \item If $y\in B$, then $y^{p^n}$ is in the image of $\phi$.
  \end{itemize}
\end{definition}

We also have the following result of Henn (Theorem A.$4$ in
\cite{Hennpaper}):

\begin{prop}\label{A4 prop}
  If $G$ is a discrete group such that there exists a
  finite-dimensional contractible $G$-CW-complex $X$ with all cell
  stabilizers finite of bounded order, then there exists a uniform
  $F$-isomorphism $$\phi:H^*(G,\F)\ra\operatorname{lim}_{\mathcal{A}_p(G)^{\operatorname{op}}}
  H^*(E,\F),$$ where $\mathcal{A}_p(G)$ denotes the category with
  objects the elementary abelian $p$-subgroups $E$ of $G$, and
  morphisms the group homomorphisms which can be induced by
  conjugation by an element of $G$.
\end{prop}

Finally, we can prove the following proposition, which is a
generalization of a result of Henn (Theorem A.$8$ in
\cite{Hennpaper}):

\begin{prop}
  Let $G$ be a group of finite virtual cohomological dimension with cohomology almost everywhere finitary over $\F$. Then
  $G$ has finitely many conjugacy classes of elementary abelian
  $p$-subgroups.
\end{prop}

\begin{proof}
  As $G$ has finite virtual cohomological dimension, there is a
  finite dimensional model, say $X$, for the classifying space
  $\eg$ for proper actions (Exercise
  \S VIII.$3$ in \cite{Brown}). Thus, $X$ is a finite dimensional
  contractible $G$-CW-complex with all cell stabilizers finite of
  bounded order, so it follows from
  Proposition \ref{A4 prop} that there is a uniform
  $F$-isomorphism $$\phi:H^*(G,\F)\ra\operatorname{lim}_{\mathcal{A}_p(G)^{\operatorname{op}}}
  H^*(E,\F).$$   Now assume that there are infinitely many conjugacy classes of
  elementary abelian $p$-subgroups of $G$. As the order of
  the finite subgroups is bounded, this means that there must be
  infinitely many maximal elementary abelian $p$-subgroups of $G$
  of the same rank $k$ (although $k$ itself need not necessarily
  be maximal). Following Henn's argument, we can use this fact to
  construct infinitely many linearly independent non-nilpotent
  classes in the inverse limit in some degree (for the details, see the proof of Theorem A.$8$ in \cite{Hennpaper}).
  Now, raising these to a large enough power and using the fact
  that $\phi$ is a uniform $F$-isomorphism, we see that
  $H^*(G,\F)$ is infinite-dimensional as an $\F$-vector space in
  some degree $m$ such that $H^m(G,-)$ is finitary over $\F$. This
  gives a contradiction to Lemma \ref{f.d. vect sp lemma}.

\end{proof}

\subsection{Proof of (ii) $\Rightarrow$ (iii)}$ $

This is immediate.

\subsection{Proof of (iii) $\Rightarrow$ (i)}$ $

Let $G$ be a group of finite virtual cohomological dimension, such
that $G$ has finitely many conjugacy classes of elementary abelian
$p$-subgroups and the normalizer of every non-trivial elementary
abelian $p$-subgroup of $G$ has cohomology almost everywhere
finitary over $\F$. We shall show that $G$ has cohomology almost
everywhere finitary over $\F$.

Firstly, let $\mathcal{A}_p(G)$ denote the poset of all the
non-trivial elementary abelian $p$-subgroups of $G$, and let
$\mathcal{S}_p(G)$ denote the poset of all the non-trivial finite
$p$-subgroups of $G$. We see from Remark 2.3(i) in
\cite{WebbPosets} that the inclusion of posets
$\mathcal{A}_p(G)\hookrightarrow\mathcal{S}_p(G)$ induces a
$G$-homotopy equivalence
$$|\mathcal{A}_p(G)|\simeq_G|\mathcal{S}_p(G)|$$
between the $G$-simplicial complexes.

Next, we need the following result (Proposition $2.7$ \S II in
\cite{Bredonbook}):

\begin{prop}\label{Bredon prop}

Let $X$ and $Y$ be $G$-CW-complexes, and $\phi:X\ra Y$ be a
$G$-equivariant cellular map. Then $\phi$ is a $G$-homotopy
equivalence if and only if $\phi^H:X^H\ra Y^H$ is a homotopy
equivalence for all subgroups $H$ of $G$.

\end{prop}

We can now prove the following key lemma:

\begin{lemma}\label{el ab contr lemma}

The complex $|\mathcal{A}_p(G)|^E$ is contractible for all
$E\in\mathcal{A}_p(G)$.

\end{lemma}

\begin{proof}

We follow an argument similar to the proof of Lemma $2.1$ in
\cite{Deltapaper}:

If $H\in\mathcal{S}_p(G)^E$, then $EH$ is a $p$-subgroup of $G$.
We can therefore define a function
$$f:\mathcal{S}_p(G)^E\ra\mathcal{S}_p(G)^E$$ by $f(H)=EH$, so for
all $H\in\mathcal{S}_p(G)^E$ we have: $$H\leq f(H)\geq E.$$ We
then see that $\mathcal{S}_p(G)^E$ is conically contractible in
the sense of Quillen (see \S $1.5$ in \cite{Quillenpaper}), which
implies that $|\mathcal{S}_p(G)^E|$ is contractible by Quillen's
argument. Finally, by Proposition \ref{Bredon prop}, we see that
$$|\mathcal{A}_p(G)|^E\simeq
  |\mathcal{S}_p(G)|^E=|\mathcal{S}_p(G)^E|,$$ and the result now follows.

\end{proof}

Finally, we can now prove the following:

\begin{theorem}
  Let $G$ be a group of finite virtual cohomological dimension. If
  $G$ has finitely many conjugacy classes of elementary abelian
  $p$-subgroups, and the normalizer of every non-trivial
  elementary abelian $p$-subgroup of $G$ has cohomology almost
  everywhere finitary over $\F$, then $G$ has cohomology almost
  everywhere finitary over $\F$.
\end{theorem}

\begin{proof}
  Let $\mathcal{A}_p(G)$ be the poset of all non-trivial
  elementary abelian $p$-subgroups of $G$, and let
  $|\mathcal{A}_p(G)|$ denote its realization as a $G$-simplicial
  complex. As $G$ has finitely many conjugacy classes of elementary
  abelian $p$-subgroups, there must be a bound on
  their orders, and so $|\mathcal{A}_p(G)|$
  is finite-dimensional, say $\dim|\mathcal{A}_p(G)|=r$. By
  Proposition \ref{n-1 connected Y}, we can embed
  $|\mathcal{A}_p(G)|$ into an $r$-dimensional $G$-CW-complex $Y$
  which is $(r-1)$-connected, such that $G$ acts freely outside
  $|\mathcal{A}_p(G)|$. The augmented cellular chain complex of
  $Y$ then gives the following exact sequence of $\Z G$-modules: $$0\ra\widetilde{H}_r(Y)\ra
  C_r(Y)\ra\cdots\ra C_0(Y)\ra\Z\ra 0,$$ which gives the following exact
  sequence of $\F G$-modules: $$0\ra\widetilde{H}_r(Y)\otimes\F\ra
  C_r(Y)\otimes\F\ra\cdots\ra C_0(Y)\otimes\F\ra\F\ra 0.$$

  In order to show that $G$ has cohomology almost everywhere
  finitary over $\F$, it is enough by an easy generalization of Lemma
  \ref{exact seq of aef mods lemma}
  to show that the functors $\EXT^*_{\F G}(\widetilde{H}_r(Y)\otimes\F,-)$ and $\EXT^*_{\F G}(C_l(Y)\otimes\F,-)$, $0\leq l\leq r$, are finitary in all
  sufficiently high dimensions.

  Firstly, notice that for every
  $E\in\mathcal{A}_p(G)$, $Y^E=|\mathcal{A}_p(G)|^E$, and hence is contractible, as the
  copies of $G$ we have added in the construction of $Y$ have free orbits, and so have no fixed points under $E$.
  Therefore, an easy generalization of Proposition \ref{why is reduced homology projective} shows that
  $\widetilde{H}_r(Y)\otimes\F$ is projective as an $\F E$-module for all
  elementary abelian $p$-subgroups $E$ of $G$.

  Let $K$ be a finite subgroup of $G$, so $\widetilde{H}_r(Y)\otimes\F$
  restricted to $K$ is an $\F K$-module with the property
  that its restriction to every elementary abelian $p$-subgroup of
  $K$ is projective. It then follows from Chouinard's Theorem \cite{Chouinardpaper} that
  $\widetilde{H}_r(Y)\otimes\F$ is projective as an $\F K$-module. As this
  holds for every finite subgroup $K$ of $G$, it then follows from Proposition \ref{finite projective dimension} that
  $\widetilde{H}_r(Y)\otimes\F$ has finite projective dimension over $\F
  G$. Hence $\EXT^n_{\F G}(\widetilde{H}_r(Y)\otimes\F,-)=0$, and thus
  is finitary, for all sufficiently large $n$.

  Next, for each $0\leq l\leq r$, consider the functor $\EXT^*_{\F
  G}(C_l(Y)\otimes\F,-)$. Provided that $n\geq 1$, we see that
  $$\begin{array}{rcl}
  \EXT^n_{\F G}(C_l(Y)\otimes\F,-) & \cong & \EXT^n_{\F G}(C_l(|\mathcal{A}_p(G)|)\otimes\F,-) \\
  \blah & \cong & \EXT^n_{\F G}(\F|\mathcal{A}_p(G)|_l,-), \\
  \end{array}$$ where $|\mathcal{A}_p(G)|_l$ consists of all the $l$-simplicies
  $$E_0<E_1<\cdots <E_l$$ in $|\mathcal{A}_p(G)|$. As $G$ acts on $|\mathcal{A}_p(G)|_l$, we can therefore split
  $|\mathcal{A}_p(G)|_l$ up into its $G$-orbits, where the
  stabilizer of such a simplex is $\bigcap_{i=0}^l N_G(E_i)$. We
  then obtain the following isomorphism: $$\begin{array}{lcl}
  \EXT^n_{\F G}(\F|\mathcal{A}_p(G)|_l,-) & \cong & \EXT^n_{\F
  G}(\F[\coprod_{\mathscr{C}} \bigcap_{i=0}^l N_G(E_i)\backslash
  G],-) \\ \blah & \cong & \prod_{\mathscr{C}}\EXT^n_{\F G}(\F[\bigcap_{i=0}^l N_G(E_i)\backslash G],-) \\
  \blah & \cong & \prod_{\mathscr{C}} H^n(\bigcap_{i=0}^l N_G(E_i),-), \\
  \end{array}$$ where the product is taken over a set $\mathscr{C}$
  of representatives of conjugacy classes of non-trivial elementary abelian
  $p$-subgroups of $G$. As we are assuming that $G$ has only
  finitely many such conjugacy classes, this product is finite.

  Now, for each $l$-simplex $E_0<E_1<\cdots<E_l$ we have
  $$C_G(E_l)\leq\bigcap_{i=0}^l N_G(E_i)\leq N_G(E_l),$$ and so
  $$|N_G(E_l):\bigcap_{i=0}^l N_G(E_i)|<\infty.$$ Then, as $N_G(E_l)$ has cohomology almost everywhere finitary over $\F$, we
  see from
  an easy generalization of Lemma \ref{finite index} that $\bigcap_{i=0}^l N_G(E_i)$ has
  cohomology almost everywhere finitary over $\F$, and so for all
  sufficiently large $n$, $H^n(\bigcap_{i=0}^l N_G(E_i),-)$ is
  finitary over $\F$. Therefore, for all sufficiently large $n$, $\EXT^n_{\F G}(C_l(Y)\otimes\F,-)$ is isomorphic to a finite product of finitary functors, and
  hence is finitary, which completes the proof.
\end{proof}

\section{A Question of Leary and Nucinkis}\label{sec:L and N}

In \cite{LearyNucinkis}, Leary and Nucinkis posed the following
question: If $G$ is a group of type $\vfp$ over $\F$, and $P$ is a
$p$-subgroup of $G$, is the centralizer $C_G(P)$ of $P$
necessarily of type $\vfp$ over $\F$?

In this section, we shall give a positive answer to this question.
Firstly, recall (see \S $2$ of \cite{LearyNucinkis}) that a group
$G$ is said to be of type $\vfp$ over $\F$ if and only if it has a
subgroup of finite index which is of type $\fp$ over $\F$.

\begin{prop}\label{VFP implies fd proper space prop}
  Let $G$ be a group which has a subgroup $H$ of finite index with $\cd_{\F} H<\infty$.
  Then there exists a finite dimensional
  $G$-CW-complex $X$ with finite cell stabilizers such that $$0\ra C_r(X)\otimes\F\ra\cdots\ra
  C_0(X)\otimes\F\ra\F\ra 0$$ is an exact sequence of $\F
  G$-modules.
\end{prop}

\begin{proof}
  As $\cd_{\F} H<\infty$, it follows from an easy generalization
  of Theorem $7.1$ \S VIII in \cite{Brown} that there exists a finite
  dimensional free $H$-CW-complex $X'$ with the property that
  $\widetilde{C}_*(X')\otimes\F$ is exact. Set $$X:=\Hom_H(G,X').$$ An easy generalization of the proof of
  Theorem $3.1$ \S VIII in \cite{Brown} then shows that $X$ has
  the required properties.
\end{proof}

Next, we prove the following key lemma, which is a variation on
Proposition \ref{finite projective dimension}:

\begin{lemma}\label{proj on rest to fin subgps lemma}
  Let $G$ be a group of type $\vfp$ over $\F$, and $M$ be an $\F
  G$-module. If $M$ is projective as an $\F K$-module for
  all finite subgroups $K$ of $G$, then $M$ has finite projective
  dimension over $\F G$.
\end{lemma}

\begin{proof}
  As $G$ is of type $\vfp$ over $\F$, we see from Proposition \ref{VFP implies fd proper space prop}
  that there exists a finite-dimensional $G$-CW-complex $X$ with finite cell stabilizers, such that $$0\ra C_r(X)\otimes\F\ra\cdots\ra
  C_0(X)\otimes\F\ra\F\ra 0$$ is exact. Now, for each $k$, $C_k(X)$ is
  a permutation module,
  $$C_k(X)\cong\bigoplus_{\sigma\in\Sigma_k}\Z[G_{\sigma}\backslash G],$$
  where $\Sigma_k$ is a set of $G$-orbit representatives of
  $k$-cells in $X$, and $G_{\sigma}$ is the stabilizer of
  $\sigma$. If we tensor the above exact sequence with $M$, then we
  obtain the following: $$0\ra M\otimes_{\F}(C_r(X)\otimes\F)\ra\cdots\ra
  M\otimes_{\F}(C_0(X)\otimes\F)\ra M\ra 0,$$ where, for each $k$,
  we have $$M\otimes_{\F} (C_k(X)\otimes\F) \cong  \bigoplus_{\sigma\in\Sigma_k} M\otimes_{\F G_{\sigma}}\F
  G.$$ Now, as $M$ is projective as an $\F G_{\sigma}$-module, we see that $M\otimes_{\F G_{\sigma}}\F G$
  is projective as an $\F G$-module. Therefore, $M\otimes_{\F}
  (C_k(X)\otimes\F)$ is projective as an $\F G$-module, and so the
  above exact sequence is a projective resolution of $M$, and we
  then conclude that $M$ has finite projective dimension over $\F
  G$.
\end{proof}

We can now answer Leary and Nucinkis' question in the case where
$P$ has order $p$. This is a variation on Lemma \ref{infinitary
vcd result}:

\begin{prop}\label{order p case prop}
  Let $G$ be a group of type $\vfp$ over $\F$, and let $P$ be a
  subgroup of $G$ of order $p$. Then $C_G(P)$ is of type $\vfp$
  over $\F$.
\end{prop}

\begin{proof}
  As $G$ is of type $\vfp$ over $\F$, we can choose a normal subgroup $N$
  of finite index which is of type $\fp$ over $\F$. Let $H:=NP$, so $H$ is
  of type $\vfp$ over $\F$.

  Next, let $\mathcal{A}_p(H)$ denote the set of all non-trivial
  elementary abelian $p$-subgroups of $H$, so $\mathcal{A}_p(H)$
  consists of subgroups of order $p$. Now $H$ acts on this set by
  conjugation, so the stabilizer of any $E\in\mathcal{A}_p(H)$ is
  simply $N_H(E)$. Also, for each $E\in\mathcal{A}_p(H)$, we see
  that the set of $E$-fixed points $\mathcal{A}_p(H)^E$ is simply
  the set $\{E\}$. We then have the following short exact sequence of $\F
  H$-modules: $$J\mono\F\mathcal{A}_p(H)\stackrel{\epsilon}{\epi}\F,$$ where
  $\epsilon$ denotes the augmentation map, and we see that for each
  $E\in\mathcal{A}_p(H)$, $J$ is free as an $\F E$-module with
  basis $\{E'-E:E'\in\mathcal{A}_p(H)\}$. Therefore, if $K$ is any
  finite subgroup of $H$, we see that $J$ restricted to $K$ is an
  $\F K$-module such that its restriction to every elementary
  abelian $p$-subgroup of $K$ is free. It then follows from
  Chouinard's Theorem \cite{Chouinardpaper} that $J$ is projective as an $\F K$-module.
  As this holds for every finite subgroup $K$ of $H$, it follows
  from Lemma \ref{proj on rest to fin subgps lemma} that $J$ has
  finite projective dimension over $\F H$.

  An argument similar to the proof of Theorem \ref{split infinitary result}
  then shows that $N_H(P)$ has cohomology almost
  everywhere finitary over $\F$. Hence, $$C_H(P)\cong P\times C_N(P)$$
  has cohomology almost everywhere finitary over $\F$, and by Lemma \ref{direct product vcd
  case}, $C_N(P)$ is of type $\fpinfty$ over $\F$. Finally, as
  $$\cd_{\F} C_N(P)\leq\cd_{\F} N<\infty,$$ we see that $C_N(P)$
  is of type $\fp$ over $\F$, and the result now follows.
\end{proof}

We can now answer Leary and Nucinkis' question. This is a
variation on Theorem \ref{p-subgroup case}:

\begin{theoremE}
  Let $G$ be a group of type $\vfp$ over $\F$, and $P$ be a
  $p$-subgroup of $G$. Then the centralizer $C_G(P)$ of $P$ is
  also of type $\vfp$ over $\F$.
\end{theoremE}

\begin{proof}
  If $P$ is trivial, then the
  result is immediate. Assume, therefore, that $P$ has order
  $p^k$, where $k\geq 1$. We proceed by induction on $k$:

  If $k=1$, then the result follows from Proposition \ref{order p case
  prop}.

  Suppose now that $k\geq 2$. Choose a subgroup
  $E\leq\zeta(P)$ of order $p$. Then $C_G(E)$ is of type $\vfp$ over $\F$ by Proposition \ref{order p case
  prop}, and so $C_G(E)$ has a normal subgroup $N$ of finite index
  which is of type $\fp$ over $\F$. Let $H:=NE$. Then $C_G(E)/E$ has the subgroup
  $H/E$ of finite index, with $H/E\cong N$ of type $\fp$ over $\F$, and so $C_G(E)/E$ is of type $\vfp$ over $\F$. By induction,
  the centralizer of $P/E$ in $C_G(E)/E$ is of type $\vfp$ over
  $\F$. Hence the normalizer of $P/E$ in $C_G(E)/E$, which is
  $$(N_G(P)\cap C_G(E))/E,$$ is of type $\fpinfty$ over $\F$, and
  by Proposition $2.7$ in \cite{Homdim} we see that $N_G(P)\cap
  C_G(E)$ is of type $\fpinfty$ over $\F$. Then, as $$C_G(P)\leq
  N_G(P)\cap C_G(E)\leq N_G(P),$$ we conclude that $C_G(P)$ is
  also of type $\fpinfty$ over $\F$.

  Now, as $G$ is of type $\vfp$ over $\F$, it has a subgroup
  $S$ of finite index which is of type $\fp$ over $\F$. Therefore, $C_S(P)$ is of type $\fpinfty$ over $\F$,
  and as $$\cd_{\F} C_S(P)\leq\cd_{\F} S<\infty,$$ we see that $C_S(P)$ is of type
  $\fp$ over $\F$, and the result now follows.
\end{proof}

\section{Groups Possessing a Finite Dimensional Model for
$\eg$}\label{sec:EG}

In this short section we consider groups possessing a finite
dimensional model for the classifying space $\eg$ for proper
actions. The proof of Theorem \ref{finitary nilpotent result}
generalizes immediately to give us the following:

\begin{propF}
  Let $G$ be a group which possesses a finite dimensional model
  for the classifying space $\eg$ for proper actions. If
  \begin{enumerate}
    \item $G$ has finitely many conjugacy classes of finite
    subgroups; and

    \item The normalizer of every non-trivial finite subgroup of
    $G$ has cohomology almost everywhere finitary,
  \end{enumerate}

  Then $G$ has cohomology almost everywhere finitary.
\end{propF}

However, the converse is false. In fact, it is false even for the
subclass of groups of finite virtual cohomological dimension, as
we shall now show. We need the following result of Leary (Theorem
$20$ in \cite{LearyVF}):

\begin{prop}
  Let $Q$ be a finite group not of prime power order. Then there is a
  group $H$ of type $\ff$ and a group $G=H\rtimes Q$ such that $G$
  contains infinitely many conjugacy classes of subgroups
  isomorphic to $Q$ and finitely many conjugacy classes of other
  finite subgroups.
\end{prop}

As $H$ is of type $\ff$, it has a finite Eilenberg--Mac Lane
space, say $Y$. As the universal cover $\widetilde{Y}$ of $Y$ is
contractible, its augmented cellular chain complex is an exact
sequence of $\Z H$-modules: $$0\ra C_n(\widetilde{Y})\ra\cdots\ra
C_0(\widetilde{Y})\ra\Z\ra 0,$$ and as $Y$ has only finitely many
cells in each dimension, we see that each $C_k(\widetilde{Y})$ is
finitely generated.

Hence, we see that $H$ has finite cohomological dimension, and is
of type $\fpinfty$. Therefore, $G$ is a group of finite virtual
cohomological dimension which is of type $\fpinfty$, and hence has
cohomology almost everywhere finitary, but $G$ does \emph{not}
have finitely many conjugacy classes of finite subgroups, which
gives us a counter-example to the converse of Proposition F above.

\section{Proof of Lemma $\ref{finite vcd lemma}$}\label{sec:Lemma
1.2}

\subsection{Proof of (i) $\Rightarrow$ (ii)}$ $

\begin{prop}
  Let $G$ be a locally (polycyclic-by-finite) group such that
  there is a finite dimensional model for $\eg$ and there is a
  bound on the orders of the finite subgroups of $G$. Then $G$ has
  finite virtual cohomological dimension.
\end{prop}

\begin{proof}
  Let $X$ be a finite dimensional model for $\eg$, and let $r=\dim X$.
  Then, for each $k$, $C_k(X)$ is a permutation module,
  $$C_k(X)\cong\bigoplus_{\sigma\in\Sigma_k}\Z[G_{\sigma}\backslash
  G],$$ where $\Sigma_k$ is a set of $G$-orbit representatives of
  $k$-cells in $X$, and $G_{\sigma}$ is the stabilizer of
  $\sigma$, so in particular each $G_{\sigma}$ is finite. Then
  $$C_k(X)\otimes\Q\cong\bigoplus_{\sigma\in\Sigma_k}\Q\otimes_{\Q
  G_{\sigma}}\Q G,$$ and as $\Q G_{\sigma}$ is semisimple, $\Q$ is
  a projective $\Q G_{\sigma}$-module. Hence, $\Q\otimes_{\Q G_{\sigma}}\Q G$ is a
  projective $\Q G$-module, and so $C_k(X)\otimes\Q$ is a projective $\Q
  G$-module. We then have that $$0\ra
  C_r(X)\otimes\Q\ra\cdots\ra C_0(X)\otimes\Q\ra\Q\ra 0$$ is a
  projective resolution of the trivial $\Q G$-module, and hence that $G$
  has finite rational cohomological dimension.

  Now, according to Hillman and Linnell \cite{HillmanLinnell},
  the Hirsch length of an
  elementary amenable group is bounded above by its rational
  cohomological dimension, so we conclude that $G$ has finite Hirsch
  length.

  Next, let $\tau(G)$ denote the unique largest locally finite normal
  subgroup of $G$. As there is a bound on the orders of the finite
  subgroups of $G$, we see that $\tau(G)$ must be finite.

  Then, as $G$ is an elementary amenable group of finite Hirsch
  length, it follows from a result of Wehrfritz \cite{WehrfritzElAm}
  that $G/\tau(G)$ has a poly-(torsion-free abelian) characteristic
  subgroup of finite index. Hence, $G$ has a poly-(torsion-free
  abelian) characteristic subgroup, say $S$, of finite index. We see
  that $S$ has finite Hirsch length, and hence finite cohomological
  dimension. We then conclude that $G$ has finite virtual cohomological dimension,
  as required.
\end{proof}

\subsection{Proof of (ii) $\Rightarrow$ (iii)}$ $

We begin by proving the following lemma:

\begin{lemma}\label{first cohomology finite}

Let $Q$ be a finite group, and $A$ be a $\Z$-torsion-free $\Z
Q$-module of finite Hirsch length. Then $H^1(Q,A)$ is finite.

\end{lemma}

\begin{proof}
  As $Q$ is finite, $H^1(Q,A)$ has exponent dividing
  the order of $Q$ (Corollary $10.2$ \S III in \cite{Brown}).
  We have the following short exact sequence: $$A\stackrel{|Q|}{\mono}
  A\stackrel{\pi}{\epi} A/|Q|A.$$ Passing to the long exact
  sequence in cohomology, we obtain the following monomorphism:
  $$H^1(Q,A)\stackrel{\pi^*}{\mono} H^1(Q,A/|Q|A).$$
  Now, as $A/|Q|A$ has finite exponent and finite Hirsch length, it is
  finite. It then follows that $H^1(Q,A)$ is finite.
\end{proof}

Next, note that all groups of finite virtual cohomological
dimension possess a finite dimensional model for $\eg$ (Exercise
\S VIII.$3$ in \cite{Brown}), so it suffices to prove the
following:

\begin{prop}

Let $G$ be a locally (polycyclic-by-finite) group of finite
virtual cohomological dimension. Then $G$ has finitely many
conjugacy classes of finite subgroups.

\end{prop}

\begin{proof}

As $G$ has finite virtual cohomological dimension, it must have a
bound on the orders of its finite subgroups. Therefore, the same
argument as in the previous subsection shows that $G$ is a
poly-(torsion-free abelian)-by-finite group of finite Hirsch
length. We proceed by induction on the Hirsch length $h(G)$ of
$G$.

If $h(G)$=1, then $G$ has a torsion-free abelian normal subgroup
$A$ of finite Hirsch length such that $G/A=Q$ is finite. There is
a 1-1 correspondence between the conjugacy classes of complements
to $A$ in $G$ and $H^1(Q,A)$ (Result 11.1.3 in \cite{RobGroup}).
Therefore, by Lemma \ref{first cohomology finite}, we see that $G$
has finitely many conjugacy classes of finite subgroups.

Suppose $h(G)>1$. We know that $G$ has a torsion-free abelian
normal subgroup $A$ of finite Hirsch length. As $h(G)>h(G/A)$, we
see by induction that $G/A$ has finitely many conjugacy classes of
finite subgroups. Let $F$ be a finite subgroup of $G$, so $AF$
lies in one of finitely many conjugacy classes, say those
represented by $AK_1,\ldots,AK_m$. Then, as each $H^1(K_i,A)$ is
finite, there are only finitely many conjugacy classes of
complements to $A$ in $AK_i$, and $F$ must lie in one of those.

\end{proof}

\subsection{Proof of (iii) $\Rightarrow$ (i)}$ $

Let $G$ be a group with finitely many conjugacy classes of finite
subgroups. Then it is clear that there must be a bound on the
orders of its finite subgroups.

\bibliographystyle{amsplain}

\bibliography{Finitary}

\providecommand{\bysame}{\leavevmode\hbox to3em{\hrulefill}\thinspace}
\providecommand{\MR}{\relax\ifhmode\unskip\space\fi MR }
\providecommand{\MRhref}[2]{%
  \href{http://www.ams.org/mathscinet-getitem?mr=#1}{#2}
}
\providecommand{\href}[2]{#2}
\begin{thebibliography}{10}

\bibitem{AccessibleCategories}
Ji{\v{r}}{\'{\i}} Ad{\'a}mek and Ji{\v{r}}{\'{\i}} Rosick{\'y}, \emph{Locally
  presentable and accessible categories}, London Mathematical Society Lecture
  Note Series, vol. 189, Cambridge University Press, Cambridge, 1994.

\bibitem{Homdim}
Robert Bieri, \emph{Homological dimension of discrete groups}, second ed.,
  Queen Mary College Mathematical Notes, Queen Mary College Department of Pure
  Mathematics, London, 1981.

\bibitem{Bredonbook}
Glen~E. Bredon, \emph{Equivariant cohomology theories}, Lecture Notes in
  Mathematics, No. 34, Springer-Verlag, Berlin, 1967.

\bibitem{BrownPaper}
Kenneth~S. Brown, \emph{Homological criteria for finiteness}, Comment. Math.
  Helv. \textbf{50} (1975), 129--135.

\bibitem{Brown}
\bysame, \emph{Cohomology of groups}, Graduate Texts in Mathematics, vol.~87,
  Springer-Verlag, New York, 1982.

\bibitem{Chouinardpaper}
Leo~G. Chouinard, \emph{Projectivity and relative projectivity over group
  rings}, J. Pure Appl. Algebra \textbf{7} (1976), no.~3, 287--302.

\bibitem{stronggpgradedring}
Jonathan Cornick and Peter~H. Kropholler, \emph{Homological finiteness
  conditions for modules over strongly group-graded rings}, Math. Proc.
  Cambridge Philos. Soc. \textbf{120} (1996), no.~1, 43--54.

\bibitem{Hennpaper}
Hans-Werner Henn, \emph{Unstable modules over the {S}teenrod algebra and
  cohomology of groups},  \textbf{63} (1998), 277--300.

\bibitem{HillmanLinnell}
J.~A. Hillman and P.~A. Linnell, \emph{Elementary amenable groups of finite
  {H}irsch length are locally-finite by virtually-solvable}, J. Austral. Math.
  Soc. Ser. A \textbf{52} (1992), no.~2, 237--241.

\bibitem{Continuity:cohomologyfunctors}
Peter~H. Kropholler, \emph{Groups with many finitary cohomology functors},
  (Preprint, University of Glasgow 2007).

\bibitem{Deltapaper}
Peter~H. Kropholler and Guido Mislin, \emph{Groups acting on finite-dimensional
  spaces with finite stabilizers}, Comment. Math. Helv. \textbf{73} (1998),
  no.~1, 122--136.

\bibitem{LearyVF}
Ian~J. Leary, \emph{On finite subgroups of groups of type {VF}}, Geom. Topol.
  \textbf{9} (2005), 1953--1976 (electronic).

\bibitem{LearyNucinkis}
Ian~J. Leary and Brita E.~A. Nucinkis, \emph{Some groups of type {$VF$}},
  Invent. Math. \textbf{151} (2003), no.~1, 135--165.

\bibitem{Leinster}
Tom Leinster, \emph{Higher operads, higher categories}, London Mathematical
  Society Lecture Note Series, vol. 298, Cambridge University Press, Cambridge,
  2004.

\bibitem{Quillenpaper}
Daniel Quillen, \emph{Homotopy properties of the poset of nontrivial
  {$p$}-subgroups of a group}, Adv. in Math. \textbf{28} (1978), no.~2,
  101--128.

\bibitem{RobGroup}
Derek J.~S. Robinson, \emph{A course in the theory of groups}, Graduate Texts
  in Mathematics, vol.~80, Springer-Verlag, New York, 1993.

\bibitem{WebbPosets}
J.~Th\'{e}venaz and P.~J. Webb, \emph{Homotopy equivalence of posets with a
  group action}, J. Combin. Theory Ser. A.

\bibitem{WehrfritzElAm}
B.~A.~F. Wehrfritz, \emph{On elementary amenable groups of finite {H}irsch
  number}, J. Austral. Math. Soc. Ser. A \textbf{58} (1995), no.~2, 219--221.

\bibitem{Wiebel}
Charles~A. Weibel, \emph{An introduction to homological algebra}, Cambridge
  Studies in Advanced Mathematics, vol.~38, Cambridge University Press,
  Cambridge, 1994.

\end{thebibliography}

\end{document}